\documentclass{rmmcart}
\usepackage{amsmath,amssymb,amsfonts,enumerate,amsthm,bm,geometry}
\usepackage{amscd}
\usepackage[all]{xy}
\input xy
\xyoption{all}

\newtheorem{thm}{Theorem}[section]
    \theoremstyle{Definition}
    
    \theoremstyle{Definition and Remark}
    \newtheorem{defi-rem}[thm]{Definition and Remark}
    \newtheorem{defis-rem}[thm]{Definitions and Remark}
    \newtheorem{defis-rems}[thm]{Definitions and Remarks}
    \newtheorem{defi-Not}[thm]{Definition and Notation}
    \theoremstyle{Lemma}
    \newtheorem{lem}[thm]{Lemma}
    \theoremstyle{remark}
    \newtheorem{rem}[thm]{Remark}
    \theoremstyle{Corollary}
    
    \newtheorem{exam}[thm]{Example}
    
    \newtheorem{prop}[thm]{Proposition}
    \newtheorem{ques}[thm]{Question}

\DeclareMathOperator{\coker}{Coker}

 \DeclareMathOperator{\Ass}{Ass}

\DeclareMathOperator{\depth}{depth}

\newcommand{\et}{\text{e}} \newcommand{\Ht}{\text{H}}   
 \newcommand{\Kt}{\text{K}} \newcommand{\Et}{\text{E}}   
  \newcommand{\hit}{\text{ht}}  \newcommand{\As}{\text{Ass}}
  \newcommand{\dep}{{\rm depth}} \newcommand{\dime}{{\rm dim}}
\newcommand{\map}{\mathfrak{p}} \newcommand{\mam}{\mathfrak{m}} \newcommand{\maq}{\mathfrak{q}} \newcommand{\maa}{\mathfrak{a}}
\newcommand{\mab}{\mathfrak{b}}   \newcommand{\man}{\mathfrak{n}}

\newcommand{\homm}{\text{Hom}}

 \newcommand{\im}{\text{im}}

\newcommand{\dsum}{\bigoplus}

\newcommand{\ten}{\bigotimes}
 \newcommand{\tor}{\text{Tor}}
\newcommand{\ins}{\bigcap} 
 \newcommand{\Rc}{\widehat{R}}
 \newcommand{\Frac}{\text{Frac}} \newcommand{\Ra}{R(a^{1/2})}

\DeclareMathAlphabet{\mathcalligra}{T1}{calligra}{g}{f}

\newcommand{\llar}{-\kern-5pt-\kern-5pt\longrightarrow}
\def\restr{{\kern-1pt\restriction\kern-1pt}}

\title{Annihilators of Koszul Homologies and Almost Complete Intersections}
\author[Ehsan Tavanfar]{Ehsan Tavanfar}

\address{
	 School of Mathematics, Institute for Research in Fundamental Sciences (IPM), P. O. Box 19395-
	 5746, Tehran, Iran, and,\newline Department of Mathematics, Shahid Beheshti University, G.C., Tehran, Iran.}
 \email{tavanfar@ipm.ir and tavanfar@gmail.com}

\thanks{The author is supported  in part by a grant from IPM-Iran and in part by NSF grant  DMS 1162585. }

\date{December, 27, 2018}
\keywords{Almost complete intersection, approximation complex, canonical module, Koszul annihilator, multiplicity.}
\subjclass[2010]{13D02, 13H10, 13H15.}
\begin{document}
	
\maketitle
\begin{center}In the memory of my father, Manouchehr Tavanfar, who passed away at the time of the preparation of   this paper.\end{center}
\markboth{Ehsan Tavanfar}{Annihilators of Koszul homologies}
  \begin{abstract}
         In this article,  we propose a question on the annihilators of positive Koszul homologies of a system of parameters of an almost complete  intersection $R$. The question can be stated in terms of the acyclicity of certain (finite) residual approximation complexes whose $0$-th homologies are the residue field of $R$. We show that our question has an affirmative answer  for the  first Koszul homology of any almost complete intersection, as well as for all positive Koszul homologies of certain system of parameters which exist in some almost complete intersection rings with small multiplicities.  The statement about the first Koszul homology is shown to be equivalent to the  Monomial Conjecture and thus follows from its validity.
  \end{abstract}


\section{Introduction}  

 Hochster's Monomial Conjecture, which has been recently settled affirmatively by Yves Andre in \cite{AndreLaConjecture}, was a challenging open question in Commutative Algebra for about four decades and it has various equivalent forms.  One of them, which inspired many  results in this paper, is given by Dutta in \cite{DuttaTheMonomial}  and states that  an almost complete intersection ring $R$ and a system of parameters $\mathbf{x}$ of $R$ satisfies the inequality, $\ell_R\big(R/(\mathbf{x})\big)\gneq \ell\big(\Ht_1(\mathbf{x},R)\big)$. In particular, Dutta's Theorem reduces the Monomial Conjecture to  almost complete intersection rings. In this direction, we present Proposition \ref{AlmostSmall}, which shows that almost complete intersection rings may play a further role in the context of the homological conjectures, despite the establishment of the Monomial Conjecture.  

  In the present paper, we show that the aforementioned Dutta's inequality is equivalent to the assertion, $\big((\mathbf{x}):\mam\big)\subseteq 0:_R\Ht_1(\mathbf{x},R)$ (see, Proposition \ref{KoszulIsHard }). Bearing this equivalence in mind, the   paper addresses the following question.

 \begin{ques}\label{Q1} Let $R$ be an  almost complete intersection and $\mathbf{x}$ be a system of parameters of $R$.  Then is $\big((\mathbf{x}):\mam\big)\subseteq 0:_R\Ht_i(\mathbf{x},R)$ for each $i\ge 1$? 
 \end{ques}

Approximation complexes, as a variant of Koszul type complexes,   are introduced and investigated  in \cite{HerzogSimisKoszul}. A new generation of approximation complexes, so-called,  residual approximation complexes are invented  in \cite{HassanzadehResidual}, to establish a conjecture on the  Cohen-Macaulayness of certain residual intersections.  Then, in, \cite{HassanzadehNaelitonResidual}, the authors show that  the acyclicity of the residual approximation complexes has strong connection with  annihilators of Koszul homologies, so that the foregoing question can be rephrased, equivalently, as follows.  

\begin{ques}\label{Q2}
	Let $(R,\mam)$ be an almost complete intersection, $\mathbf{x}$ be  a system of parameters for $R$  and $z\in \big((\mathbf{x}):\mam\big)\backslash (\mathbf{x})$, i.e.\footnote{To be more precise on this equivalence, note that, by definition of the colon ideal, $z\in\big((\mathbf{x}):\mam\big)$ if and only if $\mam\subseteq \big((\mathbf{x}):z\big)$. On the other hand, $z\notin (\mathbf{x})$, precisely when, $(\mathbf{x}):z$ is a proper ideal. Combining these facts we get, $z\in\big((\mathbf{x}):\mam\big)\backslash (\mathbf{x})$ if and only if $\mam=\big((\mathbf{x}):z\big)$.} $(\mathbf{x}):z=\mam$. Then we are endowed with a residual approximation complex $\mathcal{Z}^+_\bullet\big(\mathbf{x},z\big)$ which is a finite complex consisting of Koszul cycles of  $(\mathbf{x},z)$ satisfying, $\Ht_0\bigg(\mathcal{Z}^+_\bullet\big(\mathbf{x},z\big)\bigg)=R/\mam$. The question is  whether $\mathcal{Z}^+_\bullet\big(\mathbf{x},z\big)$  resolves $R/\mam$, i.e., whether $\mathcal{Z}^+_\bullet\big(\mathbf{x},z\big)$ is an acyclic complex?
\end{ques}

Theorem \ref{AcyclicitySettled}  proves these 
questions in the affirmative in  the non-trivial case where the system of parameters $\mathbf{x}$ contains $\mam^2$ (with $x_1=p$ if, moreover, $R$ has mixed characteristic $p>0$). To this aim,  
we firstly found that our problem reduces to the case where $R$ has multiplicity $2$. Thereafter, by looking at  examples using the Macaulay2 program, it was guessed that any such an almost complete intersection has to satisfy an inequality, $$\dep(R)\ge \dim(R)-2=\dime(R)-\et(R),$$ where $\et(R)$ denotes the multiplicity of $R$. The validity of this inequality is then proved by which the above questions are answered positively.     In the beginning of  section $2$, a more detailed overview of the structure of the proof of our main result is given.

An example is presented to show that the inequality $\dep(R)\ge \dim(R)-\et(R)$ does not hold in general if we drop  the assumption $\et(R)\le 2$ on the almost complete intersection $R$.\footnote{This example is due to S. Hamid Hassanzadeh.}
We stress that  the violation of the inequality $\dep(R)\ge \dime(R)-\et(R)$ for an almost complete intersection $R$ with $\et(R)=3$ (or with higher multiplicity) does not imply that  the answer of Question \ref{Q1} (equivalently Question \ref{Q2}) is negative.

  \section{The   results}
  
  In the preliminary part of this section we introduce and fix the  notation used throughout the article, and we recall  certain theorems which exist already in the literature but are part of the backbone of the proof of the main result of our paper.  Since this is a rather long  section,  we outline the structure of the section as follows.
  
  As mentioned in the introduction, the main result is Theorem \ref{AcyclicitySettled}. It settles affirmatively the equivalent questions of the introduction in the special  case where, roughly speaking, the parameter ideal contains the second power of the maximal ideal. In the light of Proposition \ref{KoszulIsHard } whose proof exploits  Dutta's Theorem (restated in our paper as Theorem \ref{DuttaCriterion}), we show that the restriction of  Question \ref{Q1}  to the first Koszul homologies  is another form of  Hochster's Monomial Conjecture (which is a Theorem now). Thus Proposition \ref{KoszulIsHard } shows that for answering Question \ref{Q1}, in the general case and not only in  our special case,  it suffices to verify the annihilators of Koszul homologies $\Ht_i(\mathbf{x},R)$ with $i\ge 2$. The proof of  Theorem \ref{AcyclicitySettled}, is divided into two cases, in  one case 
   it is shown that if our special s.o.p. $\mathbf{x}$ is not a part of a minimal generating set of the maximal ideal $\mam$, then, in view of the Lemma \ref{HilbertBurch} in conjunction with the finite free extensions of Remark \ref{SquareOfSequence}, our almost complete intersection has Cohen-Macaulay defect at most $1$ and so we are done in this case.  The second case of the proof of Theorem \ref{AcyclicitySettled} deals with the case when $\mathbf{x}$ is a part of a minimal generating set of $\mam$.  In this case we have Proposition \ref{aci has e<=2} (deduced from Lemma \ref{lem edim-dim=2}) which together with Lemma \ref{InequalityHolds} imply that  our almost complete intersection has Cohen-Macaulay defect and multiplicity at most $2$, i.e. there are at most two non-zero positive Koszul homologies with respect to the parameter ideals. Consequently, it turns out that, we are required only to  inspect the annihilator of  $\Ht_2(\mathbf{x},R)$, for our  special s.o.p. $\mathbf{x}$. Then the proof ends after applying,  Serre's Euler Poincar\'e Characteristic description of the multiplicity in conjunction with  Dutta's equivalence, to deduce that  $\Ht_2(\mathbf{x},R)$ has to be  a $1$-dimensional vector space.
   
   After the preliminary part,  some auxiliary lemmas  are given to deduce  Lemma \ref{InequalityHolds}. Lemma \ref{InequalityHolds}, which is of intrinsic interest and yet is another application of Dutta's equivalence, settles a nice   inequality between the Cohen-Macaulay defect and multiplicity for almost complete intersections with multiplicity $\le 2$. While, in view of Proposition \ref{HassanzadehExample}, this inequality does not hold for an almost complete intersection $R$ with unknown multiplicity $3\le e(R)\le 600$, but yet Lemma \ref{InequalityHolds} besides its key application in the proof  of our main theorem has found further applications in some special cases of  Stillman's conjecture (see, \cite{KhouryOnTheProjective}\footnote{In \cite{KhouryOnTheProjective}, Lemma \ref{InequalityHolds} is cited as, Theorem 2.5, which is with respect to  the earlier  arXiv versions of our paper. The connection of Lemma \ref{InequalityHolds} with special cases of the Stillman's conjecture was firstly pointed to the author by S. H. Hassanzadeh.}).
   
  This section ends with  two examples one of them demonstrates that the class of special  almost complete intersections of Theorem \ref{AcyclicitySettled} includes non-Cohen-Macaulay instances.

      To fix the notation used in the article, let $(R,\mam,k)$ be a Noetherian local ring of dimension $d$,   $\maa$ be an  ideal of $R$ and $M$ be a finitely generated $R$-module. The notation $\mu_R(M)$ (resp. ht$_R(\mathfrak{a})$, $\ell_R(M)$), stands for the minimal number of generators of $M$ as $R$-module (resp. the height of the ideal $\maa$ of $R$, the length of $M$ as an $R$-module). Also  $\text{Soc}_R(M)=0:_M\mam\cong \text{Hom}_R(R/\mam,M)$. The injective envelop of the residue field $R/\mam$ of $R$ is denoted by $\Et_R(R/\mam)$, and the canonical module of $R$ is the,  unique up to isomorphism if exists, finitely generated $R$-module $\omega_R$ such that $\text{Hom}_R\big(\omega_R,\Et_R(R/\mathfrak{m})\big)\cong \Ht^d_\mam(R)$. When $R$ is a domain, $\text{rank}_R(M)$ denotes the $\Frac(R)$-vector space dimension of $M\otimes_R \Frac(R)$. Occasionally, when the ambient ring is fixed and  clear and no confusion is likely to arise, we remove the subscript $R$ from the aforementioned (relevant) notation. The notation $\text{V}(\maa)$ denotes the set of prime ideals of $R$ containing the ideal $\maa$. Moreover, $\text{assht}(R)=\{\map\in \Ass(R):\dim(R/\map)=\dim(R)\},$ and $$\text{assht}(\maa)=\{\map\in \text{ass}(\maa):\dim(R/\map)=\dim(R/\maa)\}.$$

        We say that $R$ is an almost complete intersection whenever $R$ is a residue ring of a regular local ring $A$ by an ideal $\maa$ such that $\maa$ can be generated minimally by  $\hit(\maa)+1$ elements (such a defining ideal, is also called an almost complete intersection ideal). The Koszul complex and Koszul homologies   of a sequence $\mathbf{x}$ with coefficients in an $R$-module $M$ are denoted by $K_{\bullet}(\mathbf{x},M)$ and $\Ht_i(\mathbf{x},M)$.
      
      If $\maa$ is an  $\mam$-primary ideal generated by a sequence $\mathbf{y}:=y_1,\ldots,y_t$, then, the Hilbert-Samuel multiplicity of $M$ with respect to the ideal, $\maa$, is the natural number,
      $$e_R(\mathbf{y},M)=e_R(\maa, M):=\lim\limits_{n\rightarrow \infty}d!\frac{\ell_R(M/\maa^nM)}{n^d}. $$ Furthermore, the notation, $e_R(M)$ (or $e(M)$ when the ambient ring is clear), stands for the Hilbert-Samuel multiplicity, abbreviated as multiplicity, of $M$ with respect to the maximal ideal $\mam$ of $R$.   We refer to the excellent book \cite{HermannIkedaEquimultiplicity} for general theory of the multiplicity. 
      
      \begin{rem} \label{MultiplicityInfiniteResidueField}
       Lots of problems on multiplicity can be reduced to the case where $R$ is a complete local ring with algebraically closed residue field. Namely, If $(S,\man)$ is a weakly unramified\footnote{That is, $\mam S=\man$.} flat $(R,\mam,k)$-algebra, then in view of the above formula, or by \cite[(5.1) Proposition]{HermannIkedaEquimultiplicity}, we have, $\et_R(M)=\et_S(M\ten_R S)$, for each finitely generated $R$-module $M$. In particular, $\et(R)=\et(\Rc)$. Once we now that $R$ is a complete local ring then we may choose a Cohen-presentation, $$R=C_k[[X_1,\ldots,X_m]]/\maa,$$ of $R$ wherein either $C_k=k$ or $C_k$ is  an unramified complete discrete valuation ring of mixed characteristic with the residue field $k$. Then there exists a faithfully flat weakly unramified embedding, $\eta:C_k\rightarrow C_\mathcal{K}$\footnote{In the mixed characteristic, the existence of $\eta $ is followed by \cite[Theorem 29.2.]{MatsumuraCommutative}. Moreover, the flatness of $C_\mathcal{K}$ over $C_k$ is a consequence of the $p$-torsion-freeness of $C_\mathcal{K}$.},
       	such that $\mathcal{K}$ is the algebraically closure of $k$, and either $C_\mathcal{K}=\mathcal{K}$ or $C_\mathcal{K}$ is an unramified complete discrete valuation ring of mixed characteristic with the residue field $\mathcal{K}$. $\eta$ promotes to a weakly unramified  flat  monomorphism, $C_k[[X_1,\ldots,X_m]]/\maa\rightarrow S:=C_\mathcal{K}[[X_1,\ldots,X_m]]/\maa^e$.\footnote{Here, the notation $\maa^e$, stands for the extension of $\maa$ to $C_\mathcal{K}[[X_1,\ldots,X_m]]$.} Consequently, $\et(R)=\et(\Rc)=\et(S)$.
     \end{rem}
 
	    Listing, as already done below, the collection of the  standard but key theorems used in the sequel on multiplicity theory, etc.,  might be handy to read the article.
	    
	    \begin{rem}\label{FabulousTheorems} Let $(R,\mam,k)$ be a local ring of dimension $d$,  $M$ be a finitely generate $R$-module and $\mathbf{x}:=x_1,\ldots,x_d$ be a s.o.p. for $R$. The following statements hold.
	    	\begin{enumerate}
	    	\item[(i)]  If $k$ is an infinite field, then by \cite[(4.15) Remark]{HermannIkedaEquimultiplicity} there exists a
	    	 s.o.p. $\mathbf{x}'$ of $R$ such that $e(\mathbf{x}',R)=e(R)$.
	    	 \item [(ii)] By virtue of \cite[Appendix II, Corollary, page 90]{SerreLocalAlgebra} for each $i\ge 0$ we are endowed with the following positivity  of  the Higher Euler-Poincar\'e Characteristic, $$\mathcal{X}_i(\mathbf{x},M):=\sum\limits_{j\ge i}(-1)^{j-i}\ell_R\big(\Ht_j(\mathbf{x},M)\big)\ge 0.$$
	    	 \item[(iii)] In the light of and beauty of \cite[Appendix II, Remark, page 90]{SerreLocalAlgebra},  the vanishing of the Higher Euler-Poincar\'e Characteristic (defined in the previous part) is equivalent to the vanishing of the relevant Koszul homologies, i.e. for $i\ge 1$ we have, $$ \mathcal{X}_i(\mathbf{x},M)=0,\ \text{if and only if},\ \Ht_j(\mathbf{x},M)=0\ \forall j\ge i.$$
	    	 \item[(iv)] Serre's formula provides the following intriguing handy description of the multiplicity with respect to the s.o.p.  $\mathbf{x}$ in terms of the Euler-Poincar\'e characteristic  of the Koszul homologies, i.e. we are endowed with the identity, $\mathcal{X}_0(\mathbf{x},M)=\sum\limits_{i=0}^d(-1)^i\ell_R\big(\Ht_i(\mathbf{x},M)\big)=\et_R(\mathbf{x},M).$ See, \cite[Theorem 1, page 57]{SerreLocalAlgebra} or \cite[Theorem 4.7.6]{BrunsHerzogCohenMacaulay} for two references to this identity.
	    	 \item[(v)] By, e.g., \cite[Theorem 16.8, page 131]{MatsumuraCommutative} Koszul complex is depth sensitive, i.e. $$ d-\text{sup}\{i:\Ht_i(\mathbf{x},M)\neq 0\}=\text{depth}_R(M).$$
	    	 \item[(vi)] (Noether normalization, the complete local case) Suppose that $R$ is complete. Only if $R$ has mixed characteristic $(0,p)$ we impose the hypothesis that $x_1=p$. By, e.g., \cite[Theorem A.20.]{BrunsHerzogCohenMacaulay} there exists a subring $C_k$, so-called the coefficient ring, of $R$ such that either $C_k=k$ (precisely when $R$ has equal characteristic) or $C_k$  is a complete unramified discrete valuation ring of mixed characteristic  with the residue field $k$ (precisely when $R$ has mixed characteristic which would be necessarily $(0,p)$ as $p$ is a parameter element) and such that $C_k\rightarrow R$ induces  isomorphism on residue fields  (the completeness of $C_k$ is not stated in \cite[Theorem A.20]{BrunsHerzogCohenMacaulay}, but it is true that we can take $C_k$ as a complete ring as well, see e.g. \cite[Theorem 29.2]{MatsumuraCommutative} or \cite[Theorem 29.3]{MatsumuraCommutative}). If $R$ is equi-characteristic then by equi-characteristic case of \cite[Theorem A.22]{BrunsHerzogCohenMacaulay}, $C_k[[x_1,\ldots,x_d]]$ is a (power series over a field) subring of $R$ over which $R$ is a module finite extension. As well, if $R$ has mixed characteristic, since $p$ is a parameter element and can be extended to a s.o.p., so one can argue as \cite[Theorem A.22(ii)]{BrunsHerzogCohenMacaulay} (which is stated only for domains)  to see that again $R$ is module finite over its subring $C_k[[x_2,\ldots,x_d]]$ which is a power series ring.
	    	\end{enumerate}
	    \end{rem}   
    
   To close the preliminary part of this section, we present the following Dutta's equivalence which is essential to the results of our article and shall be applied repeatedly. But the  next theorem has two parts.  While in the first part we copy  Dutta's Theorem verbatim, but it is quite routine to check that the second part, whose statement is  adapted to the setting of our article, is another equivalent form of the part (i).
   
   \begin{thm}\label{DuttaCriterion}(see, \cite[1.3. Proposition]{DuttaTheMonomial}) We have the following two equivalent statements.
   	\begin{enumerate}
   		\item [(i)] The Monomial Conjecture is valid for all local rings if and only if for every regular local ring $R$ and for every pair of ideals $(I,J)$ of $R$ such that i) $I$ is a complete intersection, ii) $J$ is an almost complete intersection (i.e. $J$ is minimally generated by (height $J+1$) elements, iii) height $I$+height $J=\dim\ R$ and iv) $(I+J)$ is primary to the maximal ideal of $R$, the following length inequality holds:
   	$$ \ell\big(R/(I+J)\big)>\ell\big(\tor^R_1(R/I,R/J)\big).$$
   	    \item[(ii)] (c.f. \cite[1.2 Proposition]{DuttaTheMonomial}) The Monomial Conjecture is valid for all local rings if and only if for every almost complete intersection $R$ and any system of parameters $\mathbf{x}$ of $R$, the inequality, $\ell\big(\Ht_1(\mathbf{x},R)\big)< \ell(R/(\mathbf{x}))$, holds.
   	 \end{enumerate}
   \end{thm}
    
   \medskip

    We need several auxiliary facts to prove Lemma \ref{InequalityHolds}. The first one is  a general lemma:
    \begin{lem}\label{LemmaLift} Let $(A,\man)$ be a regular local ring and $\mab=(y_1,\ldots,y_s)$ be an almost complete intersection ideal of $A$. Let $R=A/\mab$ and suppose that  $\dim(R)=d$ and $\mathbf{x}:=x_1,\ldots,x_d$  is a system of parameters of  $R$. For an arbitrary lift, $\tilde{x_1},...,\tilde{x_d}$, of $\mathbf{x}$ to $A$, we may assume, without loss of generality (up to changing the generators of $\mab$), that  $\tilde{x_1},...,\tilde{x_d},y_1,\ldots,y_{s-1}$ is a regular sequence of $A$.
    \end{lem}
    \begin{proof}
    Let $\maa:=(\tilde{x_1},\ldots,\tilde{x_d})$. Note that, $\hit(\maa)=d$,  otherwise  there exists a prime ideal $\map\in \text{V}(\maa)$ of height $\le d-1$ and a prime ideal $\maq\in \text{assht}(\mab)$ such that,  by virtue of \cite[Theorem 3., page 110]{SerreLocalAlgebra}, $$\hit_A(\maa+\mab)\le \hit_A(\map+\maq)\le \hit_A(\map)+\hit_A(\maq)< d+s-1=\dime(A/\mab)+\hit_A(\mab)=\dime(A),$$ contracting with the fact that the image of a set of generators of  $\maa$ extends to a system of parameters for $R$. Thus $\maa$ satisfies, $\hit(\maa)=\mu(\maa)=d$, and whereby it is a complete intersection. The rest of the proof of the claim is a standard method in commutative algebra based on the application of, \cite[Theorem 124., page 90]{KaplanskyCommutative}, in conjunction with the fact that, $\hit(\maa+\mab)=\dime(A)$, similar to  the solution of \cite[ 1.2.21, page 15]{BrunsHerzogCohenMacaulay}.
    \end{proof}

    The following Lemma is one of the main  ingredients of the proof of Lemma \ref{InequalityHolds}. The fact about the depth of the canonical module is known to the experts.

    \begin{lem} \label{CanonicalModule}
    	Let $R$ be an almost complete intersection of dimension $d$ and $\mathbf{x}:=x_1,\ldots,x_d$  be a system of parameters for $R$. Then the following statements hold.
    	\begin{enumerate}
    		\item[(i)] There exists an exact sequence,
    		\begin{center}
    			$0\rightarrow \Ht_2(\mathbf{x},R)\rightarrow \omega_R/(\mathbf{x})\omega_R\overset{\theta}\rightarrow \omega_{R/(\mathbf{x})R}{\rightarrow} \Ht_1(\mathbf{x},R)\rightarrow 0$.
    		\end{center}
    		
    		\item[(ii)] $\Ht_i(\mathbf{x},R)\cong \Ht_{i-2}(\mathbf{x},\omega_R)$, for each $i\ge 3$.
    		
    	    \item[(iii)]   $$\begin{cases}\dep(\omega_R)=\dep(R)+2,\ & \text{if~~} \dep(R)\le d-2\\ \dep(\omega_R)=d,\ & \text{if~~}\dep(R)\ge d-1.\end{cases}$$
    		
    		\item[(iv)] $\et(\mathbf{x},\omega_R)=\et(\mathbf{x},R).$
    		
    	\end{enumerate}
\end{lem}
	
    	\begin{proof}
    			Assume that $(A,\man)$ is a regular local ring and $\mab=(y_1,\ldots,y_s)$ is an almost complete intersection ideal of $A$ such that $R=A/\mab$. According to Lemma \ref{LemmaLift}, there exists an ideal $(\tilde{x_1},...,\tilde{x_d})A$ of $A$ which is a lift of the ideal $(x_1,...,x_d)R$ of $R$ such that  $\maa=(\tilde{x_1},\ldots,\tilde{x_d})$ is a complete intersection  and  $\tilde{x_1},\ldots,\tilde{x_d},y_1,\ldots,y_{s-1}$ forms a regular sequence of $A$. By the abuse of  notation, we use the same notation $x_1,\ldots,x_d$  to denote the lift of $\mathbf{x}$ to $A$.
    		
    		We prove (i), (ii) and (iii).  Let $\mathbf{y}'$ denotes the truncated sequence, $y_1,\ldots,y_{s-1}$. Consider the double complex
		$M_{p,q}:=\Kt_p\big(\mathbf{x};A/(\mathbf{y}')\big)\ten_A \Kt_q\big(y_s;A\big)$.
		 Note that
		 \begin{equation}
    		\label{SpectralSequnece}
    		\Ht_i\big(\text{Tot}(M)\big)\cong \Ht_i\big(\mathbf{x},y_s;A/(\mathbf{y}')\big)\cong \Ht_i\big(y_s;A/(\mathbf{x},\mathbf{y}')\big)=\begin{cases}
    		0, &  i\ge 2 \\	0:_{A/(\mathbf{x},\mathbf{y}')}y_s\overset{\text{\cite[(1.6)]{AoyamaSomeBasic}}}{\cong}\omega_{R/(\mathbf{x})}, & i=1. \\ R/(\mathbf{x}), & i=0
    		\end{cases}
    		\end{equation}

    		Furthermore, since, $\big((\mathbf{y}'):y_s\big)/(\mathbf{y}')\cong\homm_{A/(\mathbf{y}')}\big(R,A/(\mathbf{y}')\big)\overset{\text{\cite[(1.6)]{AoyamaSomeBasic}}}{\cong}\omega_R$, so we have, $$^{II}\Et^2_{p,q}=\begin{cases}\Ht_q(\mathbf{x};R),&p=0\\ \Ht_q\bigg(\mathbf{x};\Big(\big((\mathbf{y}'):y_s\big)/(\mathbf{y}')\Big)\bigg)\cong\Ht_q(\mathbf{x};\omega_R),& p=1.\\0, &p\neq 0,1\end{cases}$$ Now the desired exact sequence is just the five term exact sequence of this spectral sequence (see, \cite[Theorem 10.31 (Homology of Five-Term Exact Sequence)]{Rotman}). 
    		
    		For the second part note that according to the vanishings of (\ref{SpectralSequnece}) for $i\ge 2$, all of the maps, $$ d^2:\Ht_{i+2}(\mathbf{x};R)\rightarrow \Ht_i(\mathbf{x};\omega_R),\ (i\ge 1)$$ arising from the second page of the spectral sequence are isomorphisms.
    		
    		(iv). The exact sequence of the first part of the lemma implies that, $$\ell\big(R/(\mathbf{x})R\big)-\ell\big(\Ht_1(\mathbf{x},R)\big)=\ell\big(\omega_{R/(\mathbf{x})R}\big)-\ell\big(\Ht_1(\mathbf{x},R)\big)=\ell\big(\omega_R/(\mathbf{x})\omega_R\big)-\ell\big(\Ht_2(\mathbf{x},R)\big).$$
		The Serre's formula of the multiplicity Remark \ref{FabulousTheorems}(iv) together with Lemma \ref{CanonicalModule}(ii) imply that
    		  \begin{align*}
    		  \et(\mathbf{x},R)=\sum\limits_{i=0}^d(-1)^i\ell\big(\Ht_i(\mathbf{x},R)\big)&=\ell\big(R/(\mathbf{x})R\big)-\ell\big(\Ht_1(\mathbf{x},R)\big)+\ell\big(\Ht_2(\mathbf{x},R)\big)+\sum\limits_{i=3}^d(-1)^i\ell\big(\Ht_i(\mathbf{x},R)\big)&\\&=\ell\big(\omega_R/(\mathbf{x})\omega_R\big)-\ell\big(\Ht_2(\mathbf{x},R)\big)+\ell\big(\Ht_2(\mathbf{x},R)\big)+\sum\limits_{i=3}^d(-1)^i\ell\big(\Ht_i(\mathbf{x},R)\big)&\\&=\ell\big(\omega_R/(\mathbf{x})\omega_R\big)+\sum\limits_{i=1}^d(-1)^i\ell\big(\Ht_i(\mathbf{x},\omega_R)\big)&\\&
    		  =\et(\mathbf{x},\omega_R)
    		  \end{align*}
    	\end{proof}

   The next lemma is analogous to  \cite{HunekeARemark}. In contrast to  \cite{HunekeARemark}, our lemma does not need the equi-characteristic assumption \footnote{In \cite{HunekeARemark}, the author needed the equi-characteristic assumption to construct an appropriate regular local  subring, while such a regular subring has been assumed in the statement of our lemma.}  and the ring $R$ in our lemma is not necessarily local. Our proof is similar to \cite{HunekeARemark} with the difference that in our rank $2$ case we do not need to use the Evans-Griffith Syzygy Theorem, and instead  simpler theorems are used. We proved the following lemma  in order to deduce Lemma \ref{InequalityHolds}, but after a while we discovered that \cite{HunekeARemark}  had been proved earlier in 1982 by Huneke.

    \begin{lem} \label{vsHuneke} \label{CohenMacaulayQuadratic}
    	Suppose that $(A,\man)$ is a complete regular local ring. Let $R$ be a module finite   extension of $A$ which is a torsion-free $A$-module of (torsion-free) rank $2$. Then $R$ is Cohen-Macaulay if and only if  $R$ satisfies the Serre condition  $S_2$.
    	  \begin{proof}
    	  	 It suffices to prove that $R$ is Cohen-Macaulay provided $R$ is $S_2$.   By virtue of \cite{HochsterMcLaughlinSplitting},  we have the splitting inclusion $A\rightarrow R$; so that  $R=A\dsum I$ for some $A$-module $I$.  Since $R$ has (torsion-free) rank $2$ over $A$,  $I$ has rank $1$. In particular, we may presume that $I$ is an ideal of $A$ (because $I$ is torsion-free and a finitely generated $A$-module).  Since $R$ is $S_2$,  any part of a system of parameters, $y_1,y_2\in A$ forms a regular sequence on $R$ and thence on $I$. Consequently, $I$ is an ideal of $A$ which satisfies the $S_2$-condition as $A$-module.
    	  	  Therefore $I_\map$ is a Cohen-Macaulay (thus free and reflexive) $A_\map$-module, if  $\dep(A_\map)=\dime(A_\map)\le 1$. Also $\dep(I_\map)\ge 2$, if $\dep(A_\map)=\dime(A_\map)\ge 2$. Consequently, $I$ is a reflexive ideal of $A$ in the light of
		  \cite[Proposition 1.4.1]{BrunsHerzogCohenMacaulay}. Now the result follows from the fact that reflexive ideals of unique factorization domains are principal.
    	  	
    	  \end{proof}
    \end{lem}

 
 The following lemma will be used in the proof of Lemma \ref{InequalityHolds}. 
 
 \begin{lem}\label{Lemma torsionfree} Let $A$ be a domain of dimension $d$  and $S$ be an $A$-algebra which is a finite $A$-module. Then $S$ is unmixed (every associated prime has the same dimension) if and only if it is torsion-free  as an $A$-module.
 \end{lem}
 \begin{proof}  $S$ is torsion-free over $A$ if and only if $\As_A(S)=\{0\}$.   Assume that $S$ is unmixed.  Let $\map\in \As_A(S)$. Then $\map S\subseteq 0:_Sx$ for some $0\neq x\in S$. Thus $\map S\subseteq \maq$ for some $\maq\in \As_S(S)$. Since $S$ is unmixed and $A/(\maq\ins A)\hookrightarrow S/\maq$ is an integral extension, we have $d=\dime(S/\maq)=\dime\big(A/(\maq\ins A)\big)$ i.e. $\maq\ins A=0$. Therefore $\map=0$ as desired.

 To see the other implication, notice that $\As_A(S)=\{\maq \cap A: \maq \in \As_S(S)\}$ (see \cite[Exercise 6.7]{MatsumuraCommutative}).  Hence $S$ being $A$-torsion-free implies that $\maq \cap A=0$ for all $\maq \in \As_S(S)$. Since $S$ is integral over $A$, we have $\dime(S/\maq)= \dime(A/\maq \cap A)=\dime(A)=d$  for all $\maq \in \As_S(S)$.
 \end{proof}
   The second part of the following lemma is used in the  course of the proof of Theorem \ref{AcyclicitySettled} which  answers   Question \ref{Q1} and Question \ref{Q2} in the affirmative  (in our special case of investigation). Although the second part of the next lemma is stated for equi-characteristic rings, but the proof of Lemma \ref{InequalityHolds}(ii)  yields a mixed characteristic
version as long as there exists an appropriate Noether normalization  (see, Remark \ref{NoetherNormalizationRemark}). The first part of the next lemma  is a refinement  of \cite[Proposition 3.4]{Huneke etal aci} when the  Gorenstein ideal $J$ in the statement of \cite[Proposition 3.4]{Huneke etal aci} is a  complete intersection as well (the part (i) of the next lemma is characteristic free and it is beyond the realm of graded rings).

    \begin{lem} \label{InequalityHolds} Let $R$ be an almost complete intersection. Then
    	  $$ \et(R)\ge \dim(R)-\dep(R) $$
      in the following cases:
       \begin{enumerate}
       	 \item[(i)]   $\dim(R)\le 2$ or $\et(R)=1$,
       	 \item[(ii)] $R$ contains a field and  $\et(R)=2 $,
       	 \item[(iii)]  $R$ contains a field and $\dim(R)=3$.
       \end{enumerate}	
       \end{lem}

       \begin{proof}  In view of Remark \ref{MultiplicityInfiniteResidueField}, without loss of generality, we can assume that $R$ is complete with  infinite residue field.
      
       (i) The case where $\dim(R)\le 1$ is quite trivial. Let  $\dim(R)=2$; so that  we only need to show that  $\et(R)\ge 2$ provided $\dep(R)=0$. If $\dep(R)=0$ then $\Ht_2(\mathbf{y},R)\neq 0$ wherein $\mathbf{y}$ is any system of parameters of $R$. For a suitable  $\mathbf{y}$ as in Remark \ref{FabulousTheorems}(i), Serre's formula, Remark \ref{FabulousTheorems}(iv), states that
	  \begin{equation}\label{eq1}
       	  \et(R)=\ell\big(R/(\mathbf{y})\big)-\ell\big(\Ht_1(\mathbf{y},R)\big)+\ell\big(\Ht_2(\mathbf{y},R)\big).
	  \end{equation}
	
	  Dutta (see, Theorem \ref{DuttaCriterion}(ii)) proved that   the validity of the Monomial Conjecture implies that $ \ell\big(R/(\mathbf{y})\big) -\ell\big(H_1(\mathbf{y},R)\big)\geq 1 $. So that the result follows from (\ref{eq1}).
       	
       	   If $\et(R)=1$, then  for a suitable system of parameters $(\mathbf{y})$, by Remark \ref{FabulousTheorems}(i) and Remark \ref{FabulousTheorems}(iv), $\ell\big(R/(\mathbf{y})\big)-\ell\big(\Ht_1(\mathbf{y},R)\big)+\chi_2(\mathbf{y},R)=1$ where $\chi_2(\mathbf{y},R)=\sum_{j\ge 2}(-1)^{j-2}\ell\big(\Ht_j(\mathbf{y},R)\big)\ge 0$ by virtue of Remark \ref{FabulousTheorems}(ii).  Again by Dutta's result $\ell\big(R/(\mathbf{y})\big)-\ell\big(\Ht_1(\mathbf{y},R)\big)\geq 1$.  Hence the above  equality implies that $\chi_2(\mathbf{y},R)=0,$ therefore   $\Ht_j(\mathbf{y},R)=0$ for $j\ge 2$ by the Serre's  impressive theorem stated in Remark 2.2(iii).  Hence $\dep(R)\ge d-1$.

	   (ii). Now assume that $e(R)=2$ and $R$ contains a field  which must be $k=R/\mam$.  Again, by Remark 2.1 and Remark 2.2(i) we can assume that there exists a system of parameters $\mathbf{x}$ for $R$ such that $\et(\mathbf{x},R)=\et(R)= 2$.
	   	
	  Now, let $S$ be the $S_2$-ification of $R$ and $R^{unm}=R/U$ where $U$ is the intersection of the primary components of $R$ associated to assht$(R)$.  It is known that (see for example \cite{AoyamaGotoOnTheEndomorphism} and \cite{AoyamaSomeBasic}) $S$ is an unmixed finite $R$-module and    $\omega_R\simeq \omega_S$.  As well there is an injection $h:R^{unm}\to S$ whose cokernel has dimension at most $d-2$.

	   Since, by assumption, $R$ contains a field, there exists a regular local subring $A$  of $R^{unm}$ such that $\mathbf{x}$ forms the regular system of parameters for $A$, the residue field of $A$  is $k$ and $R^{unm}$ is a finitely generated $A$-module (see, Remark 2.2(vi)).  By Lemma \ref{Lemma torsionfree}, $R^{unm}$ is a torsion-free $A$-module. So that we may apply the projection formula of multiplicity \cite[Corollary 6.5]{HermannIkedaEquimultiplicity} in conjunction with \cite[Corollary 4.7.9]{BrunsHerzogCohenMacaulay} to deduce that
	   $$\et_A(\mathbf{x},R^{unm})=[\frac{R^{unm}}{\mam}:\frac{A}{(\mathbf{x})}]\et(\mathbf{x},R^{unm}).$$
	   Since   $R^{unm}/\mam=A/(\mathbf{x})=k$, we get $e_A(\mathbf{x},R^{unm})=\et(\mathbf{x},R^{unm})$. An application of \cite[Corollary 4.7.8]{BrunsHerzogCohenMacaulay} shows that $\et(\mathbf{x},R^{unm})=\et(R)=2$.  Hence $\et_A(\mathbf{x},R^{unm})=2$. Now, considering the structure map $h:R^{unm}\to S$  
	   we get $\et_A(\mathbf{x},S)=2$ ($\coker h$ has dimension at most $d-2$, $h$ is injective and multiplicity $\et_A(\mathbf{x},-)$ is additive on exact sequences).

	   Yet another application of the associativity formula \cite[Corollary 4.7.9]{BrunsHerzogCohenMacaulay} implies that
	   $${\rm rank}_A(S)\et(\mathbf{x},A)=\et_A(\mathbf{x},S).$$

	  Since $\mathbf{x}$ is a regular system of parameters of $A$, $\et(\mathbf{x},A)=1$; so that ${\rm rank}_A(S)=2$. Since $S$ is an unmixed and finite $A$-module, Lemma \ref{Lemma torsionfree} implies that $S$ is a torsion-free $A$-module of rank $2$. Therefore $S$ is Cohen-Macaulay according to Lemma \ref{vsHuneke}. Hence  $\omega_R(\simeq \omega_S)$ is a maximal Cohen-Macaulay $R$-module. So that Lemma \ref{CanonicalModule}(iii) implies that  $\dep(R)\ge d-2$ as desired.
	
       	  (iii) Assume that  $\dim(R)=3$ and $R$ contains a field. If $\dep(R)\ge 2$	then the statement is trivial. If $\dep(R)=1$ then the statement follows  from Dutta's Theorem \ref{DuttaCriterion}(ii) similarly as in the proof of part (i). If $\dep(R)=0$ then  we must have $\et(R)\ge 3$, otherwise in the light of the preceding part we get a contradiction with $\dep(R)\ge \dim(R)-\et(R)\ge 1$.
       \end{proof}

    \begin{rem}\label{NoetherNormalizationRemark}
      The reason for the restriction of equal characteristic in the statements of part (ii) and (iii) of Lemma \ref{InequalityHolds} is that:  an arbitrary system of parameters $\mathbf{x}$ of a mixed characteristic complete local ring $R$ does not necessarily provide  a Noether normalization $A\rightarrow R$ such that $\mathbf{x}$ is a regular system of parameters of $A$. Hence we cannot apply the proof of Lemma \ref{CohenMacaulayQuadratic} or \cite{HunekeARemark} to conclude that a commutative local  $S_2$ ring of multiplicity  $2$ is Cohen-Macaulay, in general.  In fact, to the best of our knowledge, the mixed characteristic case of Huneke's  \cite{HunekeARemark} is still an open problem. See  \cite{OcarrollOnATheorem} for a discussion on the mixed characteristic version of \cite{HunekeARemark}.
    \end{rem}
   In contrast to Lemma \ref{InequalityHolds}, the following proposition shows that the inequality $$\dep(R)\ge \dime(R)-\et(R)$$ fails for some almost complete intersection $R$    whose multiplicity is an unknown natural number in the interval  $[3,600]$. At the time of preparation of the paper, I do not know any counterexample for an almost complete intersection whose multiplicity is precisely $3$, although one might expect that such a counterexample  exists.  It is noteworthy to stress that the existence of such a counterexample (with multiplicity $3$ or higher) does not imply that   Question \ref{Q1} and Question \ref{Q2} have negative answers.

   \begin{prop}\label{HassanzadehExample}(\cite{Highprojdim}) Let  $K$ be a field  and  $p$ any positive integer.  There exists an almost complete intersection $R$ (with three relations) containing $K$ such that $\dime(R)-\dep(R) \geq  p^{p-1}-2$ and $\et(R)\leq p^4-p^2+1$. In particular, there exists an almost complete intersection $R$ such that $\et(R)\le 600$ but $\dim(R)-\dep(R)\ge 623$.
   \end{prop}
   \begin{proof}According to  \cite[Corollary 3.6]{Highprojdim}, over any field $K$ and for any positive integer $p$, there exists an ideal $I$ in a polynomial ring $S$ over $K$ with three homogeneous generators in degree $p^2$ such that ${\rm pd}(R=S/I) \geq  p^{p-1}$. The ideal $I$ in loc.cit. has codimension $2$ hence by Auslander-Buchsbaum Formula $\dime(R)-\dep(R)={\rm pd}_S(R)-2$.  The upper bound for the multiplicity is provided in \cite[Corollary 2.3]{Huneke etal aci}.
   Setting $p=5$, we will get a counter-example to the inequality $\et(R)\ge \dime(R)-\dep(R)$ with $\et(R)\le 601$. Moreover by \cite{Huneke etal aci}, $\et(R)$ cannot attain its maximum value because $R$ is not Cohen-Macaulay. So that $\et(R)\le 600$.
 \end{proof}

The structure of the annihilator of Koszul homologies is closely related to the Homological Conjectures, in particular to  the  Monomial Conjecture as one can see in the next proposition.

\begin{prop} \label{KoszulIsHard }
	Let $(R,\mam)$ be an almost complete intersection and $\mathbf{x}$ be any system of parameters  of  $R$. Then  the annihilator of the first Koszul homology with respect to $\mathbf{x}$ is not $(\mathbf{x})$  i.e.,
	$$(\mathbf{x})\subsetneq (0:_R \Ht_1(\mathbf{x};R)),$$
	if and only if the Monomial Conjecture holds. Consequently, by virtue of \cite{AndreLaConjecture}, $$((\mathbf{x}):\mam) \subseteq  (0:_R\Ht_1(\mathbf{x};R)).$$
\end{prop}
\begin{proof}
	It suffices to show that the inequality, $\ell\big(\Ht_1(\mathbf{x},R)\big)<\ell\big(R/(\mathbf{x})\big)$, in the statement of Dutta's Theorem \ref{DuttaCriterion}(ii) holds  if and only if the inclusion $\big((\mathbf{x}):\mam\big)\subseteq 0:_R\Ht_1(\mathbf{x},R)$ holds  (for an arbitrary  system of parameter $\mathbf{x}$ of an arbitrary almost complete intersection $R$).  By Lemma \ref{CanonicalModule}(i),   $\Ht_1(\mathbf{x};R)\cong \omega_{R/(\mathbf{x})}/\im (\theta)$, wherein $\theta$ is defined in the exact sequence of Lemma \ref{CanonicalModule}(i). Recall that, $\ell\big(\omega_{R/(\mathbf{x})}\big)=\ell\big(\Et_{R/(\mathbf{x})}(R/\mam)\big)=\ell\big(R/(\mathbf{x})\big)$. Therefore, in the light of the exact sequence of Lemma \ref{CanonicalModule}(i), $\ell\big(\Ht_1(\mathbf{x},R)\big)<\ell\big(R/(\mathbf{x})\big)$ is valid if and only if  $\im(\theta)\neq 0$. 
	If, $$\big((\mathbf{x}):\mam\big)\subseteq 0:_R\Ht_1(\mathbf{x},R)= 0:_R\big(\omega_{R/(\mathbf{x})}/\im(\theta)\big),$$
	then $\im(\theta)$ has to be non-zero, because it is known that canonical module of any Cohen-Macaulay  ring is always faithful, i.e. its  annihilator is the zero ideal of the ring. Hence,   $\ell\big(\Ht_1(\mathbf{x},R)\big)<\ell\big(R/(\mathbf{x})\big)$ follows from $\big((\mathbf{x}):\mam\big)\subseteq 0:_R\Ht_1(\mathbf{x},R)$. Conversely suppose that the inequality of Dutta's Theorem \ref{DuttaCriterion}(ii) holds, in other words $\im(\theta)\neq 0$. Then, as $\omega_{R/(\mathbf{x})}\cong \Et_{R/(\mathbf{x})}(R/\mathfrak{m})$,  the Matlis dual, $$\big(\Ht_1(\mathbf{x},R)^\vee\cong\big)\big(\omega_{R/(\mathbf{x})}/\im(\theta)\big)^\vee:=\homm_{R/(\mathbf{x})}\big(\omega_{R/(\mathbf{x})}/\im(\theta),E_{R/(\mathbf{x})}(R/\mam)\big),$$  is a proper ideal, say $\maa\subset R/(\mathbf{x})$, of $R/(\mathbf{x})$ which has the same annihilator as of $0:_{R/(\mathbf{x})}\Ht_1(\mathbf{x},R)$ by the property of Matlis duality. Consequently, as $\maa\subseteq \mam/(\mathbf{x})$, we have, $$\big((\mathbf{x}):\mam)/(\mathbf{x})=0:_{R/(\mathbf{x})}\big(\mam/(\mathbf{x})\big)\subseteq 0:_{R/(\mathbf{x})}\maa=0:_{R/(\mathbf{x})}\Ht_1(\mathbf{x},R),$$ as was to be proved.
\end{proof}

   From now on, we study those  almost complete intersection  rings $R$ which satisfy
     $\mathfrak{m}^2\subseteq (\mathbf{x})R$,
     for some system of parameters $\mathbf{x}:=x_1,\ldots,x_d$ of $R$ such that $x_1=p$ if, additionally, $R$ has mixed characteristic $p>0$. We  prove that the residue field of $R$ has a resolution of length $\le d+1$ by certain residual approximation complexes.  Due to the complexity of the structures, we refer to \cite{HassanzadehResidual} and \cite{HassanzadehNaelitonResidual} for  detailed explanation of the  structures of these complexes.  The motivating property to mention these complexes here is that the acyclicity of these complexes is related to the uniform annihilator of positive Koszul homologies and  to homological conjectures.

      \begin{thm}\label{HassanzadehCreativity}
      	 (\cite[ Theorem 4.4, Corollary 4.5 ]{HassanzadehNaelitonResidual}) Let $R$ be a (Noetherian) ring, and let, $\maa=(\mathbf{a})=(a_1,\ldots,a_s)$, and, $I=(\mathbf{f})=(b,a_1,\ldots,a_s)$. Then there exists a complex,
      	   $$\mathcal{Z}^+_\bullet(a_1,\ldots,a_s,b):= 0\rightarrow \mathcal{Z}^+_{s+1}\rightarrow \mathcal{Z}^+_{s}\rightarrow \cdots\rightarrow \mathcal{Z}^+_1\rightarrow R\rightarrow 0,  $$
      	   such that, $\Ht_0\big(\mathcal{Z}^+_\bullet(a_1,\ldots,a_s,b)\big)=R/(\maa:_Rb)$, and, $\mathcal{Z}^+_i=\dsum\limits_{j=i}^{s+1}Z_j(\mathbf{f})^{\oplus n_j},$ for some positive integers $n_j$, wherein, $Z_j(\mathbf{f})$, is the $j$-th module of cycles of the Koszul complex $\Kt(a_1,\cdots,a_s,b;R)$. Moreover, $\mathcal{Z}^+_\bullet(a_1,\ldots,a_s,b)$, is acyclic if and only if  $b\in \big(0:_R\Ht_i(a_1,\ldots,a_s;R)\big)$ for each natural number $i$.
      \end{thm}

       Translating the above theorem into the  setting  of its preceding paragraph, i.e. when $(R,\mam)$ is almost complete intersection and $\mathbf{x}$ is a s.o.p. of $R$ whose generated ideal contains $\mam^2$ with $x_1=p$ in the case of mixed characteristic, for each $z\in \mam\backslash (\mathbf{x})$, we obtain a complex $\mathcal{Z}^+_\bullet(\mathbf{x},z)$ consisting of Koszul cycles of the sequence $(\mathbf{x},z)$ such that $\Ht_0\big(\mathcal{Z}^+_\bullet(\mathbf{x},z)\big)=R/\mam$. In the sequel we establish the acyclicity of the this complex which provides us with a nice finite resolution of $R/\mam$. We do this, by proving that this class of almost complete intersections have multiplicity at most two, in the non-Cohen-Macaulay case\footnote{If $R$ is a Cohen-Macaulay almost complete intersection  with a s.o.p $\mathbf{x}$ whose generated ideal contains $\mam^2$, then we may have, $\et(R)=3$. For instance, let $R$ be the residue ring of, $\mathbf{Q}[X_1,X_2,X_3,X_4,X_5,X_6]$, modulo the ideal generated by size $2$ minors of the generic $2\times 3$ matrix of   indeterminates.}. So that, Lemma \ref{InequalityHolds}(ii) implies that $\dep(R)\ge d-2$. In Theorem \ref{AcyclicitySettled}, we shall see how this lower bound for the depth would imply  that the Koszul homologies of $R$ with respect to $\mathbf{x}$ are $R/\mam$-vector spaces, as required.  However, in order to accomplish this,  we  shall have need of an additional assumption on $\mathbf{x}$, i.e. $\mathbf{x}$ is, furthermore, a part of a minimal generating set for the maximal ideal of $R$. We overcome the minimality by passing to an appropriate extension which is explained in Remark \ref{SquareOfSequence}.

 \begin{rem} \label{FirstRemark}
       	Let $a\in R$. Then the free $R$-module $R\dsum R$ acquires a ring structure via the following rule,
       	$$(r,s)(r^\prime,s^\prime)=(rr^\prime+ss^\prime a,rs^\prime+r^\prime s).$$
       	We use the notation $R(a^{1/2})$ to denote the foregoing ring structure of $R\dsum R$. In fact it is easily seen that the map, $R(a^{1/2})\rightarrow R[X]/(X^2-a)$, which takes $(r,s)$ to $(sX+r)+(X^2-a)R[X]$ is an isomorphism of $R$-algebras. In particular if $R$ is an almost complete intersection then so is $R(a^{1/2})$. We are given   the  extension map $R\rightarrow R(a^{1/2})$ by the rule $r\mapsto (r,0)$ which turns $\Ra$ into a free $R$-module with the basis $\{(1,0),(0,1)\}$. Consequently this  extension is an integral extension of $R$ and it is subject to the following properties which all are easy to verify.
       	\begin{enumerate}
       		\item[(i)]  $\dime(R)=\dime\big(\Ra\big)$ and $a$ has a square root in $\Ra$, namely $(0,1)$, which we denote it by $a^{1/2}$.
       		\item[(ii)] If $a\in \mam$ then $\Ra$ is a local ring with unique maximal ideal $\mam_{\Ra}:=\mam\dsum R$.
       		\item[(iii)] \label{RootOfParameterElement} If $a,x_2,\ldots,x_d$ is a system of parameters of $R$ then $a^{1/2},x_2,\ldots,x_d$ is a system of parameters for $\Ra$.      Moreover, if $\mam^2\subseteq (a,x_2,\ldots,x_d)$, then, $$\mam_{\Ra}^2=(\mam^2+Ra)\dsum \mam\subseteq (a^{1/2},x_2,\ldots,x_d).$$
       	
       	\end{enumerate}
       	
       \end{rem}

       In the following remark, we  adjoin the square root of an arbitrary sequence of $R$, to $R$, and we obtain an  $R$-algebra  which is a finite free $R$-module and such that the  square root sequence forms a part of a minimal generating set of the maximal ideal of the new free extension. This remark  has a role in the proof of Theorem \ref{AcyclicitySettled} wherein it is applied to construct, from an almost complete intersection $(R,\mam)$ bearing a s.o.p. $\mathbf{x}$ which is not a part of a minimal generating set of $\mam$ but $\mam^2\subseteq (\mathbf{x})$,  an almost complete intersection extension $(R',\mam')$ of $(R,\mam)$  such that $R'$ admits a s.o.p. $\mathbf{x}'$ (which is the square root of $\mathbf{x}$) such that $\mathbf{x}'$ is a part of a minimal generating set of $\mam'$ and   $\mam'^2\subseteq (\mathbf{x}')$.

       \begin{rem}  \label{PromoteToMinimialBasis}
       	
       	\label{SquareOfSequence} Let, $x_1,\ldots,x_l,$ be a sequence of elements of $R$ contained in the maximal ideal of $R$.  We, inductively, construct the local ring $(R_i,\mathfrak{m}_i)$ by taking a square root of $x_{i}$ in, $R_{i-1}$, similarly as in the preceding remark. Then in, $R_l$, we have, $$x_i^{1/2}=(\underset{0\text{-th coordinate}}{\underbrace{0}},0,\ldots,0,1,0,\ldots,\underset{(2^l-1)\text{-th coordinate}}{\underbrace{0}}),$$ whose $2^{(i-1)}$-th coordinate is $1$ and others are zero.
       	
       	\item[(i)] Let $1\le j\le l$ and $0\le k\le 2^l-1$. We denote the element $(0,\ldots,0,\underset{k-\text{th coordinate}}{\underbrace{1}},0,\ldots,0)$ of $R_l$ by $e_k$. Then we have, $$e_kx_{j}^{1/2}=\begin{cases}e_{k+2^{j-1}}, & (j-1)\text{-\ th\ digit\ of\ }k\text{\ in\ base\ }2\text{ is\ 0} \\ x_je_{k-2^{j-1}}, & (j-1)\text{-\ th\ digit\ of\ }k\text{\ in\ base\ }2\text{ is\ 1}.\end{cases}$$ In order to see why this is the case we induct on the least natural number $s\ge j$ such that $k\le 2^s-1$. In the case where $s=j$ it is easily seen that the  $(j-1)$-th digit of $k$ in its $2$-th base representation is $0$ (is $1$) if and only if $k\le 2^{j-1}-1$ ($k\ge 2^{j-1}$). So an easy use of the multiplication rule of the ring $R_{j}:=R_{j-1}\bigoplus R_{j-1}$ proves the claim (Recall that $R_j$ is subring of $R_l$). Now assume that $s>j$. Then we have, $$e_kx_j^{1/2}=\Big(0,\ldots,\underset{2^{s-1}-1\text{-th\ coordinate}}{\underbrace{0}},\big(0,\ldots,0,\underset{(k-2^{s-1})\text{-th\ coordinate}}{\underbrace{1}},0,\ldots,0\big)x_j^{1/2}\Big).$$	
       	
       	Now set $k^\prime:=k-2^{s-1}$. Note that the $(j-1)$-th digit of the base $2$ representation of $k$ and $k^\prime$ are equal. Consequently the statement follows from our inductive hypothesis.
       	
       	\item[(ii)] We are going to show that for each $1\le i\le l$ the projection map $\tau_{2^{(i-1)}}:\mam_l^2\rightarrow R$, which is the projection to the $(2^{(i-1)})$-th coordinate,  is not surjective. In the case where $l=1$ we have $\mam_l^2=(\mam^2+x_1R)\bigoplus \mam$. So, we assume that $l\ge 2$ and the statement is true for smaller values of $l$. Then,
       	\begin{equation}
       	\label{Combinatorics}
       	\mam_l^2=\big(\mam_{l-1}^2+x_lR_{l-1}\big)\bigoplus \mam_{l-1}.
       	\end{equation}
       	
       	Now, if $i=l$ then $2^{l-1}$-th coordinate  of $\mam_l^2$ is just the first coordinate of, $$\mam_{l-1}=\mam\bigoplus R\bigoplus \cdots\bigoplus R.$$ Hence, clearly, $\tau_{2^{(l-1)}}$ is not surjective. On the other hand if $i\lneq l$ then by our inductive hypothesis $\tau_{2^{(i-1)}}:\mam_{l-1}^2\rightarrow R$ is not surjective which, in the light of the equality (\ref{Combinatorics}), implies the statement immediately.
       	
       	\item[(iii)] In continuation of our investigation in the previous part, we need to show, also, that the projection map $\tau_{2^{(i-1)}}:x_j^{1/2}R_{l-1}\rightarrow R$ is not surjective unless $i=j$ ($1\le i\le l-1$ and $1\le j\le l-1$). Let $(r_k)_{0\le k\le 2^{l-1}-1}\in R_{l-1}$. Then we have,
       	
       	\begin{center}
       		$(r_k)_{0\le k\le 2^{l-1}-1}x_j^{1/2}= \sum\limits_{k=0}^{2^{l-1}-1}r_ke_{k}x_j^{1/2}=\linebreak\sum\limits_{\substack{k=0\\ (j-1)-\text{th\ digit\ of\ }k\text{\ in\ base\ }2\text{\ is\ }0}}^{2^{l-1}-1}r_ke_{k+2^{j-1}}+\sum\limits_{\substack{k=0\\ (j-1)-\text{th\ digit\ of\ }k\text{\ \ in\ base\ }2\text{\ is\ }1}}^{2^{l-1}-1}r_kx_je_{k-2^{j-1}}$.
       	\end{center}
       	
       	Thus if $i<j$ then evidently $\tau_{2^{i-1}}$ is not surjective. On the other hand if $i>j$ and there exits  $0\le k\le 2^{l-1}-1$ with $k+2^{j-1}=2^{i-1}$ then $k=2^{i-1}-2^{j-1}$ which after a straightforward computation shows that the $(j-1)$-th digit of $k$ in base $2$ is $1$. This proves the non-surjectivity of  $\tau_{2^{i-1}}$.
       	
       	\item[(iv)]  By means of the arguments of the foregoing part we can, directly, conclude that, $$x^{1/2}_i\notin  \mam^2_{l-1}+(x^{1/2}_1,\ldots,\widehat{x^{1/2}_i},\ldots,x_{l-1}^{1/2},x_l)R_{l-1},\ (i\lneq l)$$ otherwise we must have  $\tau_{2^{(i-1)}}:(x^{1/2}_j)R_{l-1}\rightarrow R$ for some $1\le j\le l-1$ and $j\neq i$ or  $\tau_{2^{(i-1)}}:\mam^2_{l-1}\rightarrow R$ is surjective.
       	
       	\item[(v)] The elements $x_1^{1/2},\ldots,x_l^{1/2}$ forms a part of a minimal basis for the maximal ideal $\mam_l$ of $R_l$. Let, $(\alpha_1,\beta_1),\ldots,(\alpha_l,\beta_l)\in R_l=R_{l-1}\bigoplus R_{l-1}$ be such that, $$\sum\limits_{k=1}^{l-1}(\alpha_k,\beta_k)x_k^{1/2}+(\alpha_l,\beta_l)\underset{=x_l^{1/2}}{\underbrace{(0_{R_{l-1}},1_{R_{l-1}})}}\in \mam_l^2=\big(\mam_{l-1}^2+(x_l)R_{l-1}\big)\bigoplus \mam_{l-1}.$$
       	Then by a simple computation we get \begin{equation}
       	\label{FirstIdentity}
       	\sum\limits_{k=1}^{l-1}\alpha_kx_k^{1/2}+\beta_{l}x_l\in \mam_{l-1}^2+x_lR_{l-1},
       	\end{equation}
       	and,
       	\begin{equation}
       	\label{SecondIdentity}
       	\sum\limits_{k=1}^{l-1}\beta_kx_k^{1/2}+\alpha_l\in \mam_{l-1}.
       	\end{equation}
       	So the identity (\ref{SecondIdentity}) yields $\alpha_l\in \mam_{l-1}$  and thence $(\alpha_l,\beta_l)\in \mam_{l}$. Moreover,  for each $1\le k\le l-1$  we must have $(\alpha_k,\beta_k)\in \mam_l=\mam_{l-1}\dsum R_{l-1}$, otherwise  we get $\alpha_i\notin \mam_{l-1}$ for some $1\le i\le l-1$ which in view of the identity (\ref{FirstIdentity}) yields $x_i^{1/2}\in \mam_{l-1}^{2}+(x_1^{1/2},\ldots,\widehat{x_i^{1/2}},\ldots,x_{l-1}^{1/2},x_l)R_{l-1}$ violating  part (iv). Consequently, $(\alpha_k,\beta_k)\in \mam_l$ for each $1\le k\le l$. This implies that $x_1^{1/2}+\mam_l^2,\ldots,x_l^{1/2}+\mam_l^2$ is a linearly independent subset of $\mam_l/\mam_l^2$ over $R_l/\mam_l$ and thence $x_1^{1/2},\ldots,x_l^{1/2}$ is part of a minimal basis for $\mam_l$.
       	
       \end{rem}
  The next lemma holds for both equal characteristic and mixed characteristic cases. Here we present a proof in  mixed characteristic while the proof in    equi-characteristic zero is similar.  Furthermore, we consider the case $x_1=p^{1/2}$ as well, to  adapt the statement to the  case $2$ of the proof of Theorem \ref{AcyclicitySettled} wherein our  square root technique developed in  Remark \ref{PromoteToMinimialBasis} is applied for passing to an almost complete intersection  with the  extra assumption that $p^{1/2},x_2^{1/2},\ldots,x_d^{1/2}$ is a part of a minimal generating set of $\mam$. The proof deals only with the case where $x_1=p^{1/2}$, however, by the same proof, the lemma holds if $x_1=p$.
  \begin{lem}\label{lem edim-dim=2} Suppose that $R$ is an almost complete intersection and $\mathbf{x}$ is a s.o.p. of $R$ such that $\mam^2\subseteq (\mathbf{x})$, and $x_1=p^{1/2}$ or $x_1=p$ if $R$ has mixed characteristic. Assume in addition that the system of parameters $\mathbf{x}$   is a part of minimal generating set of $\mathfrak{m}$.  Then $\text{embdim}(R)-\dime(R)\le 2$.
  \end{lem}
  \begin{proof}
  	  Without loss of generality, we can assume that $R$ is complete. Let us use a presentation $p^{1/2},X_2,\ldots,X_d$ for the system of parameters
  $\mathbf{x}$ in $R$, where $R=(V,p^{1/2})[[X_2,\ldots,X_d,Z_1,\ldots,Z_u]]/I$ is a homomorphic image of the regular local ring
  $$(A,\man)=(V,p^{1/2})[[X_2,\ldots,X_d,Z_1,\ldots,Z_u]]$$ with $d+u=\text{embdim}(A)=\text{embdim}(R)$. So that  $\man^2\subseteq (p^{1/2},X_2,\ldots,X_d)+I$. Set  $I=(f_1,\ldots,f_l)$, wherein $l=\mu(I)$. We denote by $f^X_i$ the sum of those monomials of $f_i$ whose power of $X_j$ is non-zero for some $2\le j\le d$. Subsequently,   we set  $f^{p^{1/2}}_i$ to be the sum of those monomials of $f_i-f^X_i$ whose coefficients are multiple of $p^{1/2}$. It follows that, $f_i^Z:=f_i-f^X_i-f_i^{p^{1/2}}\in V[[Z_1,\ldots,Z_u]]$, and that the coefficients of the monomials of $f_i^Z$ are all invertible. Now, set  ${^{\ge 3}f_i^{\ Z}}$ (resp., ${^{\le  2}f_i^{\ Z}}$) to be the sum of  the monomials of $f_i^Z$ of total degree greater than or equal to $3$ (resp., less than or equal to $2$). Since $I\subseteq \man^2$,   it turns out that,  ${^{\le 2}f^{\ Z}_i}$, is an $V$-linear combination of the elements of the form $\{Z_iZ_j:1\le i,j\le u\}$ with invertible coefficients in $V$. Otherwise, ${^{\le 2}f^{\ Z}_i}$ and thence $f_i$, would have  a summand of the form $kZ_j^\alpha$ where, $\alpha\in \{0,1\}$, $1\le j\le u$ and $k\in V\backslash p^{1/2}V$. But this contradicts with $f_i\in \man^2$. In particular, ${^{\le 2}f^{\ Z}_i}\in (Z_1,\ldots,Z_u)^2$.
      	  	
    On the other hand the fact that, $Z_iZ_j\in (p^{1/2},X_2,\ldots,X_d)+(f_1,\ldots,f_l)$, yields $$Z_kZ_s=g_1p^{1/2}+\sum\limits_{i=2}^dg_iX_i+\sum\limits_{i=1}^lh_i\ {^{\ge 3}f_i^{\ Z}}+\sum\limits_{i=1}^lh_i\ {^{\le 2}f_i^{\ Z}},$$
    for each $1\le k,s\le u$ and some power series  $h_i,g_i$. Thus an elementary  computation
    shows that, $(Z_{1},\ldots,Z_{u})^{2}\subseteq(^{\le2}f_{1}^{Z},\ldots,^{\le2}f_{l}^{Z},p^{1/2})$.
    To be more precise, noticing that each monomial appearing in $^{\le2}f_{i}^{Z}$
    is of  degree $2$ with indeterminates in  $Z_{1},\ldots,Z_{u}$, and letting
    $^{2}g_{1}^{Z}$ (resp., $h_{i,0}$) be the sum of the degree
    $2$ monomials of $g_{1}$ of the form $vZ_{m}Z_{l}$ (resp., 
    be the constant coefficient of $h_{i}$), and considering the identity,
    \begin{center} $Z_{k}Z_{s}=\underset{\text{LHS}}{\underbrace{p^{1/2}\ ^{2}g_{1}^{Z}+\sum\limits_{i=1}^l h_{i,0}\ ^{\le2}f_{i}^{Z}}}+\underset{\text{RHS}}{\underbrace{\big((p^{1/2}g_{1}-p^{1/2}\ ^{2}g_{1}^{Z})+(\sum\limits_{i=1}^l h_{i}\ ^{\le2}f_{i}^{Z}-\sum\limits_{i=1}^l h_{i,0}\ ^{\le2}f_{i,}^{Z})+\sum\limits_{i=1}^l h_{i}\ ^{\ge3}f_{i}^{Z}+\sum\limits_{i=2}^d g_{i}X_{i}\big)}}$
    \end{center} 
    it is easily seen that the RHS vanishes, i.e.
    $Z_{k}Z_{s}=p^{1/2}\ ^{2}g_{1}^{Z}+\sum\limits_{i=1}^l h_{i,0}\ ^{\le2}f_{i}^{Z}$
    concluding  that $Z_{k}Z_{s}\in({}^{\le2}f_{1}^{Z},\cdots,{}^{\le2}f_{l}^{Z},p^{1/2})$ as asserted.  This in conjunction with the concluding assertion of the preceding paragraph yields $(Z_1,\ldots,Z_u)^2+(p^{1/2})=({^{\le 2}f_1^{\ Z}},\ldots,{^{\le 2}f_l^{\ Z}},p^{1/2})$.

    Since $R$ is an almost complete intersection,  we have, $$l=\mu(I)=\hit(I)+1=\dime(A)-\dime(R)+1=\text{embdim}(R)-\dime(R)+1=u+1.$$  Consequently, we get $$u+2=l+1\ge\mu\Big(\sum\limits_{i=1}^{l}({^{\le 2}f_i^{\ Z}})+(p^{1/2})\Big)=\mu\big((Z_1,\ldots,Z_u)^2+(p^{1/2})\big)=\Big(\big(u(u+1)\big)/2\Big)+1,$$ therefore  $u\le 2$ as desired.
  \end{proof}

\begin{prop}\label{aci has e<=2} Suppose that $R$ is an almost complete intersection and $\mathbf{x}$ is a s.o.p. for $R$ such that $\mam^2\subseteq (\mathbf{x})$, and $x_1=p$ or $x_1=p^{1/2}$ if $R$ has mixed characteristic. Suppose, moreover, that $\mathbf{x}$ is a part of minimal generating set of $\mathfrak{m}$. Then either $R$ is Cohen-Macaulay with $\et(\mathbf{x},R)= 3$, or else  $\et(\mathbf{x},R)\le 2$.
\end{prop}
\begin{proof} Without loss of generality, we can assume that $R$ is complete. Consider a minimal Cohen-presentation of $R$, $R=A/I$. So that, as in the proof of the previous lemma, $(A,\mathfrak{n})$ is the complete regular local
	ring $(V,\pi)[[X_{2},\ldots,X_{d},Z_{1},\ldots,Z_{u}]]$ where $\pi=p^{1/2}$
	if $x_{1}=p^{1/2}$ or $\pi=p$ otherwise,   $\text{embdim}\big((A,\man)\big)=\text{embdim}\big((R,\mathfrak{m})\big)$ and $I\subseteq \man^2$. Firstly, we have
      	  	  $$ \ell\big(R/(\mathbf{x})\big) =\ell\bigg(A/\big((\mathbf{x})+I)\big)\bigg)=\ell(\overline{A}/\overline{I})$$ wherein the notation, $\overline{\ \ }$, means modulo $(\mathbf{x})$. Secondly, $\overline{I}=\overline{\man^2},$ since $\man^2\subseteq I+(\mathbf{x})$ and $I\subseteq \man^2$ simultaneously. It turns out that,
      	  	
      	  	     $$ \ell\big(R/(\mathbf{x})\big) = \ell (\overline{A}/\overline{\man}^2)=\text{embdim}(\overline{A})+1=\text{embdim}(R)-\dime(R)+1.$$
      	  	   By Lemma \ref{lem edim-dim=2}, $\text{embdim}(R)-\dime(R)\le 2$ hence the above equality implies that $\ell\big(R/(\mathbf{x})\big)\le 3$.    So that $\et(\mathbf{x},R)\le 3$ (e.g. \cite[Theorem 14.10]{MatsumuraCommutative}). But if $\et(\mathbf{x},R)= 3$ then  $R$ is Cohen-Macaulay by \cite[Theorem 17.11]{MatsumuraCommutative}. Thus we have $\et(\mathbf{x},R)\le 2$ if $R$ is not Cohen-Macaulay.
\end{proof}
      The following lemma which will be used in the proof of Theorem \ref{AcyclicitySettled} follows from the Hilbert-Burch Theorem.

      \begin{lem}
      	\label{HilbertBurch} Suppose that $A$ is a regular local ring and $\maa$ is an ideal of $A$ of codimension $1$ minimally generated by $2$ elements. Then, $\dep(A/\maa)=\dim(A/\mathfrak{a})-1$.
      	  \begin{proof}
      	  	 Firstly, in view of, exercise \cite[2.2.28]{BrunsHerzogCohenMacaulay}, $A/\maa$ is not Cohen-Macaulay.  Set, $I:=(x_1,x_2)$. As, $\maa$ has codimension $1$ and $A$ is a unique factorization domain so there exists an  element $\alpha\in A$ such that $x_1=\alpha x'_1$ and $x_2=\alpha x'_2$ for some coprime elements $x'_1$ and $x'_2$ of $A$. In particular, $x'_1,x'_2$ is a regular sequence of $A$. Now, it is easily seen that the Hilbert-Burch  resolution, $0\rightarrow A\overset{\begin{bmatrix}
      	  	 	-x_2'& x_1'
      	  	 	\end{bmatrix}}{\longrightarrow} A^2\overset{\begin{bmatrix}
      	  	 	x_1\\ x_2
      	  	 	\end{bmatrix}}{\rightarrow}A\rightarrow 0$,  resolves $A/\maa$ and thus the statement follows from the Auslander-Buchsbaum formula.
      	  \end{proof}
      \end{lem}
      We are now ready  to state the main result of the paper. We use the terminology ``resolves $R/\mathfrak{m}$" in the next theorem to emphasize that the Hassanzadeh's residual approximation complex provides, in the situation of the subsequent theorem as well as for Cohen-Macaulay rings, or potentially in the context of the Question 1.2 if it has positive answer,  a finite acyclic  complex  whose only non-trivial homology is $R/\mathfrak{m}$. Since the residue field has infinite projective dimension unless the ambient ring is regular,  so the residual approximation complexes can be considered, occasionally, as a candidate to remedy the lack of the existence of a nice bounded acyclic complex whose $0$-th homology is $R/\mathfrak{m}$. 

      \begin{thm}\label{AcyclicitySettled}
      	Suppose that $R$ is an almost complete intersection and $\mathbf{x}$ is a s.o.p. for $R$ such that $\mam^2\subseteq (\mathbf{x})$, and $x_1=p$ if $R$ has mixed characteristic. Then for each $i\ge 1$,  $\mam\big(\Ht_i(\mathbf{x},R)\big)=0$. Consequently, for each $z\in\mam\backslash (\mathbf{x})$, the residual approximation complex $\mathcal{Z}^+_\bullet(\mathbf{x},z)$ is an acyclic complex of length $\le d+1$ \footnote{If $\depth(R)=0$ then the length of the residual approximation  complex is $d+1$, otherwise the length is $d$.} which resolves $R/ \mam$.
      	  \begin{proof}
      	  	If $R$ is a Cohen-Macaulay then there is nothing to prove, because then $\Ht_i(\mathbf{x},R)=0$ for each $i\ge 1$ and therefore the statement is obvious by virtue of Theorem \ref{HassanzadehCreativity}. So assume that $R$ is not Cohen-Macaulay (and is complete).  We separate two cases.
		
		{\bf Case 1.} If  $\mathbf{x}$ is a part of minimal generating set of $\mathfrak{m}$, then  according to Proposition \ref{aci has e<=2}, $\et(\mathbf{x},R)\le 2$.  Now if $R$ has equal characteristic then   we have $\dep(R)\ge \dime(R)-2$, by Lemma \ref{InequalityHolds}(ii). On the other hand if $R$ has mixed characteristic $p>0$ then in view of  our assumption of $x_1=p$, we are provided with a Noether normalization $A\rightarrow R$ wherein $A$ is a regular local ring with the regular system of parameters $\mathbf{x}$ (see, Remark \ref{FabulousTheorems}(vi)). Consequently, the proof of  Lemma \ref{InequalityHolds}(ii)  can  be copied verbatim to obtain the inequality, $\dep(R)\ge \dime(R)-2$.
       Now if $\dep(R)=\dime(R)-1$ then $\Ht_i(\mathbf{x},R)=0$ for each $i\ge 2$ and the statement follows from Proposition \ref{KoszulIsHard } and Theorem \ref{HassanzadehCreativity} and the discussion there; notice that $\mam\subseteq ((\mathbf{x}):\mam)$ by assumption, thus Proposition \ref{KoszulIsHard } yields $\mam\big(\Ht_1(\mathbf{x},R)\big)=0$ by which	  we can apply Theorem \ref{HassanzadehCreativity} to deduce the acyclicity of the mentioned complexes of the statement.

       So we deal with case where $\dep(R)=\dime(R)-2$. In particular, $\et(\mathbf{x},R)=2$ and $\ell\big(R/(\mathbf{x})\big)=3$ by Proposition \ref{aci has e<=2} (the length of $R/(\mathbf{x})$ is calculated in the proof of Proposition \ref{aci has e<=2}).   In this case    $\Ht_2(\mathbf{x},R)\neq 0$, hence $\ell\big(\Ht_2(\mathbf{x},R)\big)\ge 1$. As well, $\mam\Ht_1(\mathbf{x},R)=0$ by Proposition \ref{KoszulIsHard } and the assumption that $\mathfrak{m}\subseteq ((\mathbf{x}):\mam)$.  On the other hand, the  Monomial Conjecture, which is a Theorem by \cite{AndreLaConjecture} and Dutta's Theorem \ref{DuttaCriterion}(ii) imply that $\ell\big(R/(\mathbf{x})\big)-\ell\big(\Ht_1(\mathbf{x},R)\big)\ge 1$. By Serre's formula, Remark \ref{FabulousTheorems}(iv), we have
        $$\et(\mathbf{x},R)=\ell\big(R/(\mathbf{x})\big)-\ell\big(\Ht_1(\mathbf{x},R)\big)+\ell\big(\Ht_2(\mathbf{x},R)\big).$$

        Since $\et(\mathbf{x},R)=2$, we get $\ell\big(\Ht_2(\mathbf{x},R)\big)=1$. In particular $\Ht_2(\mathbf{x},R)\simeq R/\mam$; so that $\mam\Ht_2(\mathbf{x},R)=0$ as desired.

{\bf Case 2.} If  $\mathbf{x}$ is not a part of minimal generating set of $\mathfrak{m}$. In this case one may use Remark \ref{PromoteToMinimialBasis} to find  a ring $R'$ which  is module  finite and free almost complete intersection extension of $R$ bearing a s.o.p. $\mathbf{x}'$ which  is a part of a minimal basis of its maximal ideal, and $\mathbf{x}'R'$ properly contains the ideal  $\mathbf{x}R'$.  In this case $\text{embdim}(R)-\dime(R)\le \text{embdim}(R')-\dime(R')$ (This inequality holds for any flat local homomorphism, see \cite[Theorem 3.4]{MaLech}). The latter is at most $2$ according to  Lemma \ref{lem edim-dim=2}, because $\mathbf{x'}$ also contains the square  of the unique maximal ideal of $R'$ by Remark \ref{RootOfParameterElement}(iii). Since by assumption $\mathbf{x}$ is not a part of minimal generating set of $\mathfrak{m}$ we must have $\text{embdim}(R)-\dime(R)=1$.  Namely, without loss of generality we can assume that,  
\begin{equation}
 \label{Membership}
 x_d+\mam^2\in \big((x_1,\ldots,x_{d-1})+\mam^2\big)/\mam^2.
 \end{equation}
  Bearing in mind  the  construction method of $R'$ from $R$ in Remark \ref{PromoteToMinimialBasis} and following the notation therein, we have $R'=R_{d-1}(x_d^{1/2})$ and the second power of the unique maximal ideal, $\mathfrak{m}':=\mathfrak{m}_{R_{d-1}}\dsum R_{d-1}$, of $R'$, is just \begin{equation}
    \label{SecondPowerIdentity}
    \mam'^2=(\mam_{R_{d-1}}^2+x_dR_{d-1})\oplus \mam_{R_{d-1}}.
   \end{equation} 
     Since any element of the sequence $x_1,\ldots,x_{d-1}$ has a square root in $R_{d-1}$ and whence belong to $\mam_{R_{d-1}}^2$, so in view of the identity (\ref{Membership}) and (\ref{SecondPowerIdentity}), it turns out  that, $\mam'^2=\mam_{R_{d-1}}^2\oplus \mam_{R_{d-1}}$. Consequently, $$\mathfrak{m}'/\mathfrak{m}'^{2}=(\mathfrak{m}_{R_{d-1}}\oplus R_{d-1})/(\mathfrak{m}_{R_{d-1}}^{2}\oplus\mathfrak{m}_{R_{d-1}})=(\mathfrak{\mathfrak{m}}_{R_{d-1}}/\mathfrak{m}_{R_{d-1}}^{2})\oplus(R_{d-1}/\mathfrak{m}_{R_{d-1}}),$$ which implies that $\text{embdim}(R')=\text{embdim}(R_{d-1})+1$ (notice that $R_{d-1}\rightarrow R'$ induces an isomorphism on residue fields $R_{d-1}/\mam_{d-1}\rightarrow R'/\mam'$). Hence, as $\dim(R)=\dim(R_{d-1})=\dim(R')$, we get, $$2\ge \text{embdim}(R')-\dim(R')=\text{embdim}(R_{d-1})-\dim(R_{d-1})+1\ge \text{embdim}(R)-\dim(R)+1,$$
  i.e. $\text{embdim}(R)-\dim(R)=1$, as  asserted.
  
  However, by assumption, $R$ is an almost complete intersection which  is a quotient of a regular local ring $A$ by an almost complete intersection ideal $\mathbf{a}$.  Therefore $\text{embdim}(R)-\dime(R)=1$ shows that $\mathfrak{a}$ is a two generated ideal of height $1$, which by Lemma \ref{HilbertBurch} is resolved by a Hilbert-Burch  matrix and  $\dep(R)=\dime(R)-1$. Thus the result follows from Proposition \ref{KoszulIsHard }.
      	  \end{proof}
      \end{thm}

We  present an explicit non-Cohen-Macaulay example of the class of 	 rings of the previous theorem. 
      \begin{exam}
      	Set $ R=K[[Y_1,\ldots,Y_6,Z_1,Z_2]]/I,$
      	with
      	$ I=(Y_2^6Y_3^5+Z_2^2,Y_3^3Y_4^8+Z_1^2,Y_2^3Y_3^4Y_4^4+Z_1Z_2).$
      	Then  $\dep(R)=4$  and $\dime(R)=6$.
      \end{exam}
We give an example to show that, in general, the inclusion $\mathfrak{m}^2\subseteq (\mathbf{x})$ in conjunction with the non-Cohen-Macaulayness does not imply that $\et(\mathbf{x},R)\le 2$ (without assuming that $R$ 	is an almost complete intersection).
      \begin{exam} Set
      	  $ S=\mathbf{Q}[X_1,\ldots,X_4,Z_1,Z_2,Z_3],$ wherein $X_1,X_2,X_3$ have degree $1$ and $X_4,Z_1,Z_2,Z_3$ have degree $2$. Let
      	    $I:=(Z_1^2+X_3^2Z_2+X_4Z_2,Z_2^2-X_1X_2Z_3+X_4Z_3,Z_3^2,Z_1Z_2,Z_1Z_3,Z_2Z_3).$
      	  Then, in $R=S/I$ the image of the sequence $\mathbf{x}:=X_1,X_2,X_3,X_4$ forms a system of parameters  whose generated ideal contains the second power of the maximal ideal, while $\et(\mathbf{x},R)=3$. Note that $\dep(R)=3$  and $R$ is not Cohen-Macaulay.
      \end{exam}

      Inspired by Theorem \ref{AcyclicitySettled} in conjunction with Proposition \ref{KoszulIsHard }, we propose the following question, which is a generalization of the Monomial Conjecture.

      \begin{ques}\label{TheEnd} Let $R$ be an arbitrary almost complete intersection. Let, $\mathbf{x}$, be a system of parameters of $R$ and suppose that $z$ is element of $R$ whose image in $R/(\mathbf{x})$ is a non-zero element of $\text{Soc}\big(R/(\mathbf{x})\big)$. Thus we have $(\mathbf{x}):z=\mam$,  whereby, in the light of Theorem \ref{HassanzadehCreativity}, we are given  the complex, $\mathcal{Z}^+_\bullet\big(\mathbf{x},z\big)$, with, $\Ht_0\Big(\mathcal{Z}^+_\bullet\big(\mathbf{x},z\big)\Big)=R/\mam$. The question is  whether $\mathcal{Z}^+_\bullet\big(\mathbf{x},z\big)$ resolves $R/\mam$; in other words, $z\big(\Ht_i(\mathbf{x},R)\big)=0$ for each $i\ge 1$?
      \end{ques}

     \section{Small Cohen-Macaulay Conjecture}     
 
      In \cite{TavanfarReduction}, we  proved that the Small Cohen-Macaulay Conjecture reduces to excellent unique factorization domains. In order to emphasize the significance of almost complete intersections, we  end the paper by  showing that, similarly as  the Monomial Conjecture (Theorem now), the Small Cohen-Macaulay Conjecture  also reduces to  almost complete intersections. But first, we review the statement of the Small-Cohen-Macaulay Conjecture. A maximal Cohen-Macaulay $R$-module is a non-zero finitely generated $R$-module whose depth attains the largest possible  number, i.e. $\dim(R)$, in other words, a maximal Cohen-Macaulay module is a (finitely generated) Cohen-Macaulay module of dimension $\dim R$. The existence of maximal Cohen-Macaulay modules over a complete local ring (or an excellent local ring) is a long standing conjecture, so-called the Small Cohen-Macaulay Conjecture. It is conjectured by Hochster in \cite[Conjecture (6), page 10]{HochsterTopics} in the early 1970s. However, Hochster later in the 2000s, proposed a conjecture in the reverse direction, i.e. there exists a complete local ring which does not admit a maximal Cohen-Macaulay module (see, \cite[Conjecture 2]{HochsterHomological}).
      
      \begin{prop}\label{AlmostSmall}
      	The Small Cohen-Macaulay Conjecture is valid if every complete almost complete intersection has a maximal Cohen-Macaulay module.	
      \end{prop}
      \begin{proof}
      	It is well-known that the conjecture reduces to normal complete domains. So let $R$ be a normal complete local domain. Then the same argument as in \cite[1.2. Proposition]{DuttaTheMonomial} shows that there exists a complete almost complete intersection $S$ such that both of $R$ and $S$ are homomorphic image of a regular local ring $A$ and they have the same canonical module. Now, if $M$ is a maximal Cohen-Macaulay $S$-module then by \cite[Corollary 1.15.]{SchenzelOnTheUse} or \cite[Theorem 2.9.(ii)]{TavanfarTousiAStudy}\footnote{see, \cite{TavanfarTousiAStudyPublished} for the published version of the paper.}, so is the canonical module of $M$, i.e., $\omega_M:=\homm_S(M,\omega_S)=\homm_A(M,\omega_S)$. But the latter is then a maximal Cohen-Macaulay $R$-module, as $\omega_S$ is, also, the canonical module of $R$.
      \end{proof}

  \section*{Acknowledgement}
        
  A special thanks  goes to S. Hamid Hassanzadeh   for  Proposition \ref{HassanzadehExample} as well as Lemma \ref{HilbertBurch} which are due to him,  for his kind consultation which led to  the correct statement of Proposition \ref{aci has e<=2}, for teaching us about his great joint paper \cite{HassanzadehNaelitonResidual}  and finally for his efforts to improve the presentation of  earlier drafts of the paper. We also  thank Massoud Tousi,  Anurag K. Singh, Linquan Ma, Kazuma Shimomoto, Kamran Divaani-Aazar and Jason McCullough for their helpful comments and valuable suggestions. We are also grateful to the referee for his/her many valuable comments, suggestions and corrections which improved the presentation of the paper.   Part of this work was done when the author was visiting the department of mathematics of the University of Utah.  We also would like to thank the University of Utah for its hospitality during our visit in the academic year 2015-2016 as well as for its support for this research.



\end{document}